\documentclass[reqno]{amsart}
\usepackage{graphicx}
\usepackage{amssymb}
\usepackage{amsfonts}
\usepackage{indentfirst}
\usepackage{amssymb,amsmath,amsthm}
\usepackage{latexsym,bm}
\usepackage{graphicx}
\usepackage{times}
\usepackage{amsfonts}
\usepackage{indentfirst}
\usepackage[colorlinks,linkcolor=black,anchorcolor=blue,citecolor=blue]{hyperref}
\usepackage[dvipsnames]{pstricks}
\newtheorem{theorem}{Theorem}
\newtheorem{proposition}{Proposition}

\newtheorem{lemma}{Lemma}

\def\Z{\mathbb{Z}}

\def\Zp{\mathbb{Z}_{p}}
\def\Qp{\mathbb{Q}_{p}}
\def\F{\mathbb{F}}
\def\mp#1{\ ({\rm mod} \ p^{#1})}
\def\X{\mathbb{X}}

\begin{document}



\title{Dynamics of the square mapping on the ring of $p$-adic integers}


\author{Shilei Fan}

\address{School of Mathematics and Statistics, Central China Normal University, 430079, Wuhan, China}
\email{slfan@mail.ccnu.edu.cn}

\author{Lingmin Liao}

\address{LAMA, UMR 8050, CNRS,
Universit\'e Paris-Est Cr\'eteil Val de Marne, 61 Avenue du
G\'en\'eral de Gaulle, 94010 Cr\'eteil Cedex, France}
\email{lingmin.liao@u-pec.fr}

\begin{abstract}

For each prime number $p$, the dynamical behavior of the square mapping on the ring $\mathbb{Z}_p$  of $p$-adic integers is studied. For $p=2$, there are only attracting fixed points with their attracting basins. For $p\geq 3$, there are a fixed point $0$ with its attracting basin, finitely many periodic points around which there are countably many minimal components and some balls of radius $1/p$ being attracting basins. All these minimal components are precisely exhibited for different primes $p$.
\end{abstract}
\maketitle

\maketitle

\section{Introduction}

The dynamics of the quadratic maps on finite fields or rings attracts much attention in the literature (\cite{Cha84, GKRS01,Rog96,SK2006, VS04itergrap}). In particular, Rogers \cite{Rog96} studied the square mapping $f: x\mapsto x^2$ on the prime field $\Z/p\Z=\mathbb{F}_{p}$, with $p$ being a prime number.

Notice that for the square mapping, the point $0$ is fixed and one needs only to consider the points in the multiplicative group $\mathbb{F}^*_{p}:=\mathbb{F}_{p}\setminus \{0\}$.
Denote by $\varphi$ the \emph{Euler's phi function}.  For an integer $d\geq 2$, the \emph{order of $2$ modulo }$d$, which will be denoted by $ord_{d}2$, is the smallest positive integer $i$ such that
$2^i\equiv 1 ({\rm mod} \ d)$. By convention, $ord_{1}2$ is set to be $1$.
Define a directed graph $G(\mathbb{F}^*_{p})$ whose vertices are
  the elements of $\mathbb{F}^*_{p}$ and whose edges are directed from $x$ to $f(x)$ for each $x\in\mathbb{F}^*_{p}$.  Let $\sigma(\ell,k)$ be the graph consisting a cycle of length $\ell$ with a  copy of the binary tree $T_{k}$ of height $k$ attached to each vertex.
 The dynamical structure of
   the square mapping on $\mathbb{F}^*_{p}$ is described by the following theorem of Rogers \cite{Rog96}.
   \begin{theorem}[\cite{Rog96}]\label{graph}
Let $p$ be an odd prime. Put $p=2^{k}m+1$ where $m$ is odd. Then
$$G(\mathbb{F}^*_{p})=\bigcup_{d|m}   \underbrace{(\sigma(ord_{d}2,k)\cup\ldots\cup\sigma(ord_{d}2,k))}_{\varphi(d)/ord_{d}2}.$$
\end{theorem}
The graphs of $G(\mathbb{F}^*_{p})$ for $p=11$ and $17$ are depicted in Figures 1 and 2.
 \begin{figure}
  \centering
  \includegraphics[width=0.6\textwidth]{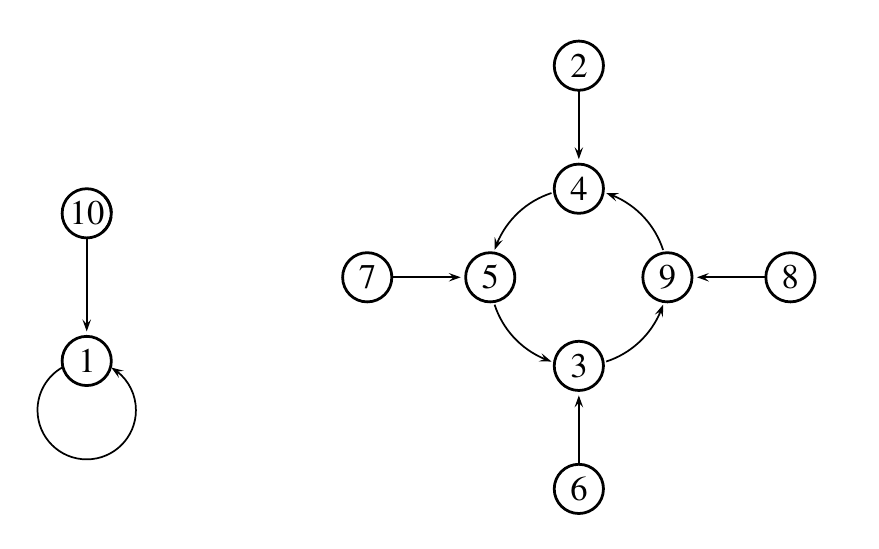}\\
  \caption{The graphs $G(\mathbb{F}^{*}_{p})$ for primes $p=11$ (thus $k=1$, $m=5$ and, $d=1$ and $5$). The vertices
  are the elements of $\mathbb{F}^{*}_{p}$ with  edges
  directed from $x$ to $x^{2}$.}
\end{figure}

\begin{figure}
  \centering
  \includegraphics[width=0.6\textwidth]{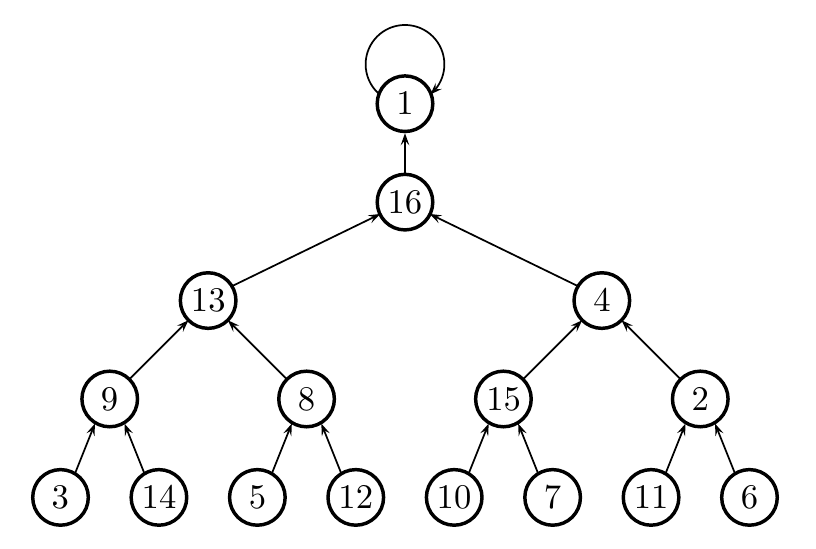}\\
  \caption{The graphs $G(\mathbb{F}^{*}_{p})$ for primes $p=17$ ($k=4, m=1$). The vertices
  are the elements of $\mathbb{F}^{*}_{p}$ with  edges
  directed from $x$ to $x^{2}$.}
\end{figure}

\medskip
In this paper, we will investigate the square mapping $f: x\mapsto x^2$ on all finite rings $\Z/p^n\Z$ and on their inverse limits $\Z_p=\underleftarrow{\lim} \Z/p^n\Z$. The space $\Z_p$ is nothing but the ring of $p$-adic integers. We are thus led to the study of the $p$-adic dynamical system $(\Z_p, f)$.

Let $(X,T)$ be a dynamical system with $X$ being a compact metric space and $T$ being a continuous map from $X$ to itself.
 For a point $x\in X$,  the \emph{orbit of $x$ under $T$} is
defined by
$$\mathcal{O}_{T}(x):=\{T^{n}(x):n\geq 0\}.$$
If $E\subset X$ is a $T$-invariant (i.e., $T(E)\subset E$) compact subset,
then $(E, T)$ is a subsystem of $(X,T)$.  The subsystem $(E,T)$ is called {\it minimal} if
 $E$ is equal to the closure $\overline{\mathcal{O}_{T}(x)}$ for
each $x\in E$. We refer to the book of Walters \cite{WalGTM79} for dynamical terminology.

For a prime number $p$,  denote by $\mathbb{Q}_p$ the field of $p$-adic numbers. Then the ring $\Zp$ of $p$-adic integers is the local ring of $\mathbb{Q}_p$. The absolute value on $\mathbb{Q}_p$ is denoted by $|\cdot|_p$. With this non-Archimedean absolute value, $\Zp$ is the unit ball of $\mathbb{Q}_p$ which is both compact and open.  For more details on $p$-adic numbers, one could consult Robort's book \cite{Rob-GTM198}.

Let $f\in\Zp[x]$ be a polynomial with coefficients in $\Zp$. Then $f$ defines a dynamical system on $\Zp$, denoted by $(\Zp,f)$. In the literature the minimality of $f$ on the whole space $\Zp$ 
is widely studied (\cite{Ana94,Ana02,Ana06,AKY11,DP09,Jeong2013, Yur13}). However, if the system is not minimal on $\Zp$, what the dynamical structure of $f$ looks like? To answer this question, one is led to do a minimal decomposition of the space $\Zp$, i.e., to find all the minimal subsystems (minimal components) of $f$.

In general, it is proved by Fan and Liao \cite{FL11} that a polynomial dynamical system $(\Zp,f\in \Zp[x])$ admits at most countablely many minimal subsystems and the polynomial system has a\emph{ minimal decomposition}.
\begin{theorem}[\cite{FL11}]\label{thm-decomposition}
 Let $f \in \mathbb{Z}_p[x]$ with degree
at least $2$. We have the following decomposition
$$
     \mathbb{Z}_p = \mathcal{P} \bigsqcup \mathcal{M} \bigsqcup \mathcal{B}
$$
where $\mathcal{P}$ is the finite set consisting of all periodic points of
$f$, $\mathcal{M}= \bigsqcup_i \mathcal{M}_i$ is the union of all (at most countably
many) clopen invariant sets such that each $\mathcal{M}_i$ is a finite union
of balls and each subsystem $f: \mathcal{M}_i \to \mathcal{M}_i$ is minimal, and each
point of $\mathcal{B}$ lies in the attracting basin of $\mathcal{P}\sqcup \mathcal{M}$.
\end{theorem}

 The minimal decomposition in Theorem \ref{thm-decomposition} was first discovered by Coelho and Parry \cite{CP11Ergodic} for the multiplications, and by Fan, Li, Yao and Zhou \cite{FLYZ07} for the affine polynomials. For the polynomials with higher order, the minimal decomposition seems hard to obtain. In \cite{FL11}, Fan and Liao succeeded in making the minimal decomposition for all quadratic polynomials but only for the prime $p=2$. Recently, Fan, Fan, Liao and Wang \cite{FFLW2013} also studied the minimal decomposition of the homographic maps on the projective line over the field $\Qp$ of $p$-adic numbers.

Furthermore, in \cite{FL11}, the authors also described the dynamics of each minimal subsystem. Let $(p_s)_{s\geq 1}$ be a sequence of positive integers such that $p_s|p_{s+1}$ for every $s\geq 1$. We denote by $\mathbb{Z}_{(p_{s})}$ the inverse limit of $\mathbb{Z}/p_{s}\Z$, which is called an \emph{odometer}. The sequence  $(p_s)_{s\geq 1}$ is called the \emph{structure sequence} of  $\mathbb{Z}_{(p_{s})}$.  The map $x\rightarrow x+1$ defined on $\mathbb{Z}_{(p_{s})}$ will be called the \emph{adding machine} on  $\mathbb{Z}_{(p_{s})}$. 
\begin{theorem}[\cite{FL11}]\label{structure-minimal}
Let $f \in \mathbb{Z}_p[x]$ with
degree at least $2$. If $E$ is a minimal clopen invariant set of $f$,
then $f : E \to E$ is conjugate to the adding machine on an odometer
$\mathbb{Z}_{(p_s)}$, where  $$(p_s) = (k, kd, k dp, kdp^2,
\cdots)$$ with integers $k$ and $d$ such that $1 \leq k\leq p$ and $d|
(p-1)$.
\end{theorem}

In this paper, we fully study the square mapping $f: x\mapsto x^2$ on $\Zp$. For any prime $p\geq 2$, the complete minimal decomposition for the system $(\Zp,x^2)$ is obtained. The structure sequences of the minimal subsystems are given.

\medskip



By Anashin \cite{Ana94, Ana02}, the dynamical structure of a polynomial on $\mathbb{Z}_p$ is derived from the structures of the induced systems on $\Z/p^{n}\Z$. In desJardins and Zieve \cite{DZunpu} and Fan and Liao \cite{FL11}, a method to study the structures on $\Z/p^{n}\Z$ inductively is developped. This method then allows us to do minimal decompositions for polynomials by knowing their dynamical structures at first levels. In particular, one needs to know, at least, the dynamical structure of the induced dynamics on $\Z/p\Z$ (i.e., at level $1$).  

For the case of the square mapping $f:x\mapsto x^2$, however, the dynamical structure at level $1$ has already been described by Rogers \cite{Rog96} (Theorem \ref{graph} at the beginning of this section).
 Hence, doing the minimal decomposition of  the square mapping $f$ on $\Zp$ will be possible.

For $a\in \Zp$ and $r>0$, denote
$D_{r}(a):=\{x\in\Zp: |x-a|_{p}<r\}$ ,  $\overline{D}_{r}(a):=\{x\in\Zp: |x-a|_{p}\leq r\}$ and $S_{r}(a):=\{x\in\Zp: |x-a|_{p}=r\}$.
Without difficulty, we can check that by iterations of $f$,  the points in $D_1(0)$ are attracted to the fixed point $0$, which means that $D_1(0)\setminus \{0\}\subset \mathcal{B}$. It is also easy to see that for the case $p=2$, all points in $D_1(1)$ are attracted to the fixed point $1$. So we have $\mathcal{P}=\{0,1\}, \mathcal{M}=\emptyset$ and $\mathcal{B}=\Z_p\setminus \{0,1\}$

For $p\geq 3$, we have seen that $0\in \mathcal{P}$ is a fixed point with $D_1(0)\setminus \{0\}=p\Zp\setminus \{0\} \subset \mathcal{B}$ as its attracting basin.  By Theorem \ref{graph}, at level $1$, $\mathbb{F}_p^*$ is a union of cycles with some binary trees of the same height attached to each vertex of the cycles. Each point in $\mathbb{F}_p^*$ is a ball of radius $1/p$. Let $\mathcal{C}\subset \Zp \setminus p\Zp$ be the union of balls corresponding to the points in the cycles and $\mathcal{T}=(\Zp \setminus p\Zp)\setminus \mathcal{C}$ be the union of balls corresponding to the points in the trees. Then $\mathcal{T}$ are attracted to $\mathcal{C}$, which means that $\mathcal{T}\subset \mathcal{B}$.
Hence, we will only treat the system $f$ restricted on $\mathcal{C}$.

For two integers $m$ and $n$, we denote by $(m,n)$ their greatest common divisor. The following minimal decomposition theorem of the square mapping $f$ on $\mathcal{C}$ is our main result. It gives a whole picture of the dynamical structure of the square mapping on $\Zp$.

\begin{theorem}\label{squringdecompsition}
Let $p$ be an odd prime with $p=2^{k}m+1$ where $m$ is an odd integer. Then $\mathcal{C}$ can be decomposed as the union of $m$ periodic points and countably many minimal components around each periodic orbit.

Let $P_m$ be the set of periodic points, i.e., $$P_{m}=\{x\in \mathcal{C}: f^{n}(x)=x ~~\text{ for some integer  $n\geq 1$}\}.$$ Then $P_m\subset \mathcal{P}$ and we can decompose $P_{m}$ in the following way: $$P_{m}=\bigsqcup_{d|m}\underbrace{\hat{\sigma}(ord_{d}2)\sqcup\cdots\sqcup\hat{\sigma}(ord_{d}2)}_{\varphi(d)/ord_{d}2}$$
 where $\hat{\sigma}(\ell)$ is a periodic orbit of period $\ell$.

Let $\hat{\sigma}(\ell)=(\hat{x}_{1},\cdots,\hat{x}_{\ell})$ be one of the periodic orbits of period $\ell$. Around this periodic orbit, we have the following decomposition
 $$\bigsqcup_{1\leq i\leq \ell}D_{1}(\hat{x}_{i})=\{\hat{x}_{1},\cdots,\hat{x}_{\ell}\}\sqcup \left(\bigsqcup_{n\geq 1}\bigsqcup_{1\leq i\leq \ell}S_{p^{-n}}(\hat{x}_{i})\right).$$
For each $n\geq 1$, the set $\bigsqcup_{1\leq i\leq \ell}S_{p^{-n}}(\hat{x}_{i})$ belongs to the minimal part $\mathcal{M}$ and contains  $\frac{(p-1)\cdot (ord_{p}2,\ell)}{ord_{p}2}\cdot p^{v_{p}(2^{p-1}-1)-1}$ minimal components, and each minimal component is a union of $j:=\frac{\ell\cdot ord_{p}2}{(ord_{p}2,\ell)}$ closed disks of radius $p^{-n-v_{p}(2^{p-1}-1)}$.

For  each  minimal component $\mathcal{M}_i$ lying in $\bigsqcup_{1\leq i\leq \ell}D_{1}(\hat{x}_{i})$, the subsystem  $f:\mathcal{M}_i\rightarrow \mathcal{M}_i$ is conjugate to the adding machine on the odometer $\Z_{(p_s)}$, where
  $$(p_s)=(\ell,\ell j,\ell j p,\ell j p^2,\cdots).$$
 \end{theorem}

\medskip
Our paper is organized as follows. In Section \ref{induceddynamics}, we study the induced dynamics on $\mathbb{Z}/p^{n}\mathbb{Z}$. Section \ref{preliminary-NT} gives some facts in Number Theory. The minimal decomposition is completed in Section \ref{decomp}. Finally in Section \ref{examples}, some examples for special primes like Fermat primes and Wieferich primes are discussed.

\bigskip
\section{Induced dynamics on $\mathbb{Z}/p^{n}\mathbb{Z}$}\label{induceddynamics}
Let $p\geq 3$ be a prime and let $f\in \mathbb{Z}_p[x]$ be a polynomial with coefficients in $\mathbb{Z}_p$. The dynamics of $f$ on $ \mathbb{Z}_p$ is determined by those of its induced finite dynamics on  $\mathbb{Z}/p^n\mathbb{Z}$ (\cite{Ana94, Ana02}). The idea to study these finite dynamics inductively comes from desJardins and Zieve \cite{DZunpu}. It allows Fan and Liao \cite{FL11} to give the decomposition theorem (Theorem \ref{thm-decomposition}) for any polynomial in $\mathbb{Z}_{p}[x]$. In this section, we will give some basic definitions and facts which are useful in proving our main theorem.  For details, see \cite{FL11} or \cite{FLpre}.

 Let $n\ge 1$ be a positive integer. Denote by $f_n$ the induced
mapping of $f$ on $\mathbb{Z}/p^n\mathbb{Z}$, i.e.,
$$f_n (x ({\rm mod}\ p^{n}))= f(x) \ ({\rm mod} \ p^{n}).$$
The dynamical behaviors of $f$ are linked to those of
$f_n$.

\begin{lemma}[\cite{Ana06,CFF09}]\label{minimal-part-to-whole}
Let $f\in \mathbb{Z}_p[x]$ and $E\subset \mathbb{Z}_p$ be a compact
$f$-invariant set. Then $f: E\to E$ is minimal if and only if $f_n:
E/p^n\mathbb{Z}_p \to E/p^n\mathbb{Z}_p$ is minimal for each $n\ge
1$.
\end{lemma}
By Lemma \ref{minimal-part-to-whole}, to study the minimality of $f$, we need to study the minimality of each $f_n$. Moreover, it is important to investigate the conditions under which the minimality of $f_{n}$ implies that of $f_{n+1}$.

Assume that $\sigma=(x_1, \cdots, x_k) \subset
\mathbb{Z}/p^n\mathbb{Z}$ is a {\it cycle} of $f_n$ of length $k$ (also
called a {\it $k$-cycle}) at level $n$, i.e.,
$$ f_n(x_1)=x_2, \cdots, f_n(x_i)=x_{i+1}, \cdots, f_n(x_k)=x_1.$$
 Let
$$ X_\sigma:=\bigsqcup_{i=1}^k  X_i \ \ \mbox{\rm where}\ \
X_i:=\{ x_i +p^nt+p^{n+1}\mathbb{Z}; \ t=0, \cdots,p-1\} \subset
\mathbb{Z}/p^{n+1}\mathbb{Z}.$$
 Then
\[
f_{n+1}(X_i) \subset X_{i+1}  \ (1\leq i \leq k-1) \ \ \mbox{\rm
and}\ \  f_{n+1}(X_k) \subset X_1.\]

Let $g:=f^k$ be the $k$-th iterate of $f$, then we have $g_{n+1}(X_{i})\subset X_{i}$ for all $1\leq i \leq k$.
In the following we shall study the behavior of the finite dynamics
$f_{n+1}$ on the $f_{n+1}$-invariant set $X_\sigma$ and determine all
cycles of $f_{n+1}$ in $X_\sigma$, which will be  called  {\it lifts} of
$\sigma$ (from level $n$ to level $n+1$). Remark that the length of any lift of
$\sigma$ is a multiple of $k$.

Let $$\mathbb{X}_i:=x_{i}+p^{n}\Zp =\{x\in \Zp: x\equiv x_i \ ({\rm mod} \ p^n)\}$$ be the closed disk of radius $p^{-n}$  corresponding  to $x_i\in \sigma$ and
  $$\mathbb{X}_{\sigma}:=\bigsqcup_{i=1}^{k} \mathbb{X}_i$$ be the clopen set corresponding to the cycle $\sigma$.

 For $x\in \mathbb{X}_{\sigma}$, denote
\begin{eqnarray}
& &a_n(x):=g'(x)=\prod_{j=0}^{k-1} f'(f^j(x)) \label{def-an} \\
& &b_n(x):=\frac{g(x)-x}{p^n}=\frac{f^k(x)-x}{p^n}.\label{def-bn}
\end{eqnarray}
The $1$-order Taylor Expansion of $g$ at $x$
\begin{eqnarray*}
 g(x+p^n t) \equiv g(x)+g^{\prime}(x)\pi^{n}t \quad ({\rm mod} \ p^{2n}),  \quad \text{for} \ t\in \{0,\dots, p-1\}
 \end{eqnarray*}
implies
\begin{eqnarray}\label{linearization}
 g(x+p^n t) \equiv x+p^n b_n(x) + p^n a_n(x) t
 \quad ({\rm mod} \ p^{2n}), \quad \text{for} \ t\in \{0,\dots, p-1\}.
 \end{eqnarray}
Define an affine
map
$$\Phi(x,t)=b_n(x)+a_n(x) t \qquad (x \in \mathbb{X}_\sigma, t \in \{0,\dots, p-1\}).$$
We usually consider the function $\Phi(x,\cdot)$ as an induced function from $\mathbb{Z}/p\mathbb{Z}$ to
 $\mathbb{Z}/p\mathbb{Z}$ by taking ${\rm mod} \ p$ and we keep the notation $\Phi(x,\cdot)$ if there is no confusion.
An important consequence of the  formula (\ref{linearization}) shows that $g_{n+1}:
X_i \to X_i$ is conjugate to the linear map $$ \Phi(x, \cdot):
\mathbb{Z}/p\mathbb{Z} \to \mathbb{Z}/p\mathbb{Z},
$$ for  $x\in \X_{i}$.
It is called the {\em linearization} of $g_{n+1}: X_i \to X_i$.

As proved in Lemma 1 of \cite{FL11}, the coefficient $a_n(x)$ (mod $p$) is always constant on
$\X_i$ and the coefficient $b_n(x)$ (mod $p$) is also constant on
$\X_i$ but under the condition $a_n(x)\equiv 1$ (mod $p$).
For simplicity, sometimes we write $a_n$ and $b_n$
without mentioning $x$.




From the values of $a_n$ and $b_n$, one can predict the behaviors of $f_{n+1}$ on $X_{\sigma}$. The linearity of the map $\Phi=\Phi(x, \cdot)$ is the key to what follows:\\
 \indent {\rm (a)} If $a_n \equiv 1 \ ({\rm mod} \ p)$ and
 $b_n \not\equiv 0 \ ({\rm mod} \ p)$, then $\Phi$ preserves a single cycle of length $p$, so that
 $f_{n+1}$ restricted to $X_{\sigma}$ preserves a single cycle of length $pk$. In this case we say $\sigma$ {\it grows}.\\
 \indent {\rm (b)} If $a_n \equiv 1 \ ({\rm mod} \ p)$ and
 $b_n \equiv 0 \ ({\rm mod} \ p)$, then $\Phi$ is the identity, so $f_{n+1}$ restricted to $X_{\sigma}$
 preserves
 $p$ cycles  of length $k$. In this case we say $\sigma$ {\it splits}. \\
 \indent {\rm (c)} If $a_n \equiv 0 \ ({\rm mod} \ p)$, then $\Phi$ is constant, so $f_{n+1}$ restricted to $X_{\sigma}$
 preserves
 one cycle of length $k$ and  the remaining points of $X_{\sigma}$ are mapped into this cycle.
  In this case we say $\sigma$ {\it grows tails}. \\
  \indent {\rm (d)} If $a_n \not\equiv0, 1 \ ({\rm mod} \ p)$, then $\Phi$ is a permutation
  and the $\ell$-th iterate of $\Phi$ reads
 \begin{eqnarray*}
\Phi^\ell(t) = b_n(a_n^{\ell}-1) /(a_n-1) +a_n^{\ell} t
\end{eqnarray*}
so that $$ \Phi^{\ell}(t)-t =(a_n^{\ell}-1) \left( t+
\frac{b_n}{a_n-1}\right).$$
 Thus,
$\Phi$ admits a single fixed point $t=-b_n/(a_n-1)$, and the
remaining points  lie on cycles of length $d$, where $d$ is the
order of $a_n$ in $(\mathbb{Z}/p\mathbb{Z})^*$. So, $f_{n+1}$
restricted to $X_{\sigma}$ preserves one cycle of length $k$ and
$\frac{p-1}{d}$ cycles of length $kd$. In this case we say $\sigma$
{\it partially splits}.
\medskip

We want to see the change of nature from a cycle to
its lifts, so it is important to study the relation between $(a_n, b_n)$ and $(a_{n+1},
b_{n+1})$.  The following lemmas are useful for our study of the dynamics of the square mapping on $\mathbb{Z}_{p}$. For details see \cite{DZunpu,FL11}
%

\begin{lemma}[\cite{DZunpu}, see also \cite{FL11}, Proposition 2]\label{cycle-p>3}
Let $ p\geq 3 $ be a prime and $n\geq 2$ be an integer. If $\sigma$ is a growing cycle of $f_n$ and $\tilde{\sigma}$
is  the unique lift of $\sigma$, then $\tilde{\sigma}$ grows.
\end{lemma}

\begin{lemma}[\cite{FL11}]\label{grows-minimal}
Let $ p\geq 3 $ be a prime and $n\geq 2$ be an integer. If $\sigma$ is a growing cycle of $f_n$, then $\sigma$ produces a minimal component, i.e., the set $\mathbb{X}_\sigma$ is a minimal subsystem of $f$.
\end{lemma}
\begin{proof}
By Lemma \ref{cycle-p>3}, if $\tilde{\sigma}$ is the lift of $\sigma$, then $\tilde{\sigma}$ also grows. Applying Lemma \ref{cycle-p>3} again, the lift of $\tilde{\sigma}$ grows. Consecutively, we find that the descendants of $\sigma$ will keep on growing. (In this case, we usually say $\sigma$ {\it always grows} or {\it grows forever}.) Hence, $f_m$ is minimal on $\mathbb{X}_\sigma/ p^m \mathbb{Z}_p$ for each $m\geq n$. Therefore, by Lemma  \ref{minimal-part-to-whole}, $(\mathbb{X}_\sigma,f)$ is minimal.
\end{proof}

\bigskip
\section{Preliminary facts in Number Theory}\label{preliminary-NT}
In this section we give some preliminary facts in number theory.

The field $\mathbb{Q}_{p}$ of $p$-adic numbers always contains a cyclic subgroup of order $p-1$, defined as
$$\mu_{p-1}:=\{x\in \Qp: x^{p-1}=1\}\subset \mathbb{Z}_{p}^{\times}.$$
Here, $\mathbb{Z}_{p}^{\times}$ stands for the set of all invertible elements in $\mathbb{Z}_{p}$.

As a cyclic group, $\mu_{p-1}$ is isomorphic to the multiplicative group $\F_p^{*}$.
\begin{lemma}\label{unityroots}
When $p$ is an odd prime, the group of roots of unity in the field $\mathbb{Q}_{p}$ is $\mu_{p-1}$.

\end{lemma}

\begin{proof}
See  Proposition 1 of Section 6.7 of  \cite{Rob-GTM198}.
\end{proof}
\begin{lemma}\label{homo}
Let $p$ be an odd prime and $\mu_{p-1}$ be the group of roots of unity in the field $\mathbb{Q}_{p}$. Let $$\varepsilon:\mu_{p-1}\rightarrow \mathbb{F}_{p}^*$$ be the reduction homomorphism. Then the following graph commutates.

 \begin{center}
 \setlength{\unitlength}{1mm}
 \begin{picture}(40,40)
 \put(2,4){$\mathbb{F}^*_{p}$}
 \put(10,4){$\vector(1,0){15}$}
 \put(15,1){$x^{2}$}
 \put(28,4){$\mathbb{F}^*_{p}$}
 \put(2,26){$\mu_{p-1}$}
 \put(5,23){$\vector(0,-1){15}$}
 \put(2,15){$\varepsilon$}
 \put(10,27){$\vector(1,0){15}$}
 \put(15,28){$x^{2}$}
 \put(28,26){$\mu_{p-1}$}
 \put(30,23){$\vector(0,-1){15}$}
  \put(31,15){$\varepsilon$}

 \end{picture}
 \end{center}
\end{lemma}
\begin{proof}
Notice that $\mu_{p-1}$ and $\mathbb{F}_{p}^{*}$ are cyclic multiplicative  groups, and   $$\varepsilon:\mu_{p-1}\rightarrow \mathbb{F}_{p}^*$$ is a group automorphism. Furthermore, the square mapping on $\mu_{p-1}$ and the square mapping on $\mathbb{F}_{p}^* $ are group homomorphisms. Hence the graph commutates.
\end{proof}

For a periodic orbit $\hat{\sigma}=(\hat{x}_{1},\hat{x}_{2},\cdots,\hat{x}_{\ell})\subset \mathbb{Z}_{p}^{\times}$ and a cycle $\sigma_n=(x_{1},x_{2},\cdots,x_{\ell})\subset (\mathbb{Z}/p^n\mathbb{Z})^*$ at level $n$, of the same length $\ell$, we write $\sigma_{n} \equiv \hat{\sigma} \mp{n}$ if
$$  x_i \equiv  \hat{x}_i \ \ ({\rm mod} \ p^n) \quad \forall 1\leq i \leq \ell.$$

The following proposition is directly derived from Lemma \ref{homo}.
\begin{proposition}\label{cycleZp}
Let $p$ be an odd prime and $f: x \mapsto x^{2}$ be the square mapping. If  $\sigma_1=(x_{1},x_{2},\cdots,x_{\ell})\subset (\mathbb{Z}/p\mathbb{Z})^*$ is a cycle of the induced mapping $f_{1}$ of length $\ell$, then there exists a unique periodic orbit $\hat{\sigma}=(\hat{x}_{1},\hat{x}_{2},\cdots,\hat{x}_{\ell})\subset \mathbb{Z}_{p}^{\times}$ such that $\sigma_{1} \equiv \hat{\sigma} \mp{}$
\end{proposition}

Conversely, for a periodic orbit of $f$ in $\Zp$, there exists a corresponding  periodic  orbit of $f_1$ in $\Z/p\Z$. Furthermore, for each integer  $n\geq 1$,
there exists a corresponding  periodic  orbit of $f_n$ in $\Z/p^n\Z$. By Lemma \ref{homo}, the proof of the following proposition is evident.
\begin{proposition}\label{cycleFp}
Let $p$ be an odd prime and  and $f: x \mapsto x^{2}$ be the square mapping. Let  $\hat{\sigma}=(\hat{x}_{1},\hat{x}_{2},\cdots,\hat{x}_{\ell})\subset \mathbb{Z}_{p}^{\times}$ be a periodic orbit of $f$ of length $\ell$. Then $\ell \leq p-1$ and for each $n\geq 1$, there exists a unique cycle $\sigma_{n}\subset \Z/p^{n}\Z$ of  $f_{n}$ of length $\ell$ such that $\sigma_{n} \equiv \hat{\sigma} \mp{n}$.
\end{proposition}

The following lemma is a basic fact in Number Theory.
\begin{lemma}\label{ord}
Let $p$ be an odd prime and $\ell\geq 1$ be an integer. Then the order of $2^{\ell}$ in $(\Z/p\Z)^*$ is $\frac{ord_{p}2}{(\ell,ord_{p}2)}.$
In particular, if $2^{\ell}\equiv 1 \mp{}$, we have $ord_{p}2\mid\ell$.
\end{lemma}

\begin{proof}
Notice that $$2^{\ell\cdot \frac{ord_{p}2}{(\ell,ord_{p}2)}}=2^{ord_{p}2\cdot \frac{\ell}{(\ell,ord_{p}2)}}\equiv 1 \mp{}.$$
Thus the order of $2^{\ell}$ is no more than $\frac{ord_{p}2}{(\ell,ord_{p}2)}.$

Write $\ell=k\cdot ord_{p}2+s$ with $k\geq 0$ and $0\leq s < ord_{p}2 $.

If $2^{\ell}\equiv 1 \mp{}$, then
$$1 \equiv 2^{\ell}=2^{k\cdot ord_{p}2+s}\equiv 2^{s} \mp{}.$$ So by the definition of $ord_{p}2 $, we have $s=0$. Hence $$ ord_p 2 \mid \ell \quad \text{ and } \quad \frac{ord_{p}2}{(\ell,ord_{p}2)}=1.$$
Since $ord_{p}(2^\ell) \leq  \frac{ord_{p}2}{(\ell,ord_{p}2)}$, we conclude that
 $$ord_{p}(2^\ell) =1=\frac{ord_{p}2}{(\ell,ord_{p}2)}.$$

If $2^{\ell}\not\equiv 1 \mp{}$, then $s\neq 0$ and $(\ell, ord_{p}2)=(s, ord_{p}2)$.
 Hence for any positive integer $i<\frac{ord_{p}2}{(\ell,ord_{p}2)}=\frac{ord_{p}2}{(s,ord_{p}2)}$, we have $ord_p 2\nmid i\cdot s $.
 So we have
$$2^{i\ell}=2^{i k\cdot ord_{p}2+is}\equiv 2^{is}\not\equiv 1 \mp{}.$$ Thus we also have $$ord_{p}(2^\ell) =\frac{ord_{p}2}{(\ell,ord_{p}2)}.$$
\end{proof}

 Now we calculate the $p$-valuations, denoted by $v_{p}(\cdot)$, of some numbers. It will be useful for finding the minimal decomposition of the square mapping on $\mathbb{Z}_{p}$.
\begin{lemma}\label{vap}
Let $p$ be an odd prime. Then for all $1\leq i<p$, $$v_{p}(2^{ord_{p}2}-1)=v_{p}(2^{i\cdot ord_{p}2}-1)<p-1.$$ In particular, $$v_{p}(2^{ord_{p}2}-1)=v_{p}(2^{p-1}-1).$$
\end{lemma}
\begin{proof}
Assume that $v_{p}(2^{ord_{p}2}-1)=s\geq 1$. Then we can write $$2^{ord_{p}2}=1+p^{s}t$$ for some integer $t\geq 1$ such that $(t,p)=1$.
For all $1\leq i<p$, we have
$$(1+p^{s}t)^{i}\equiv 1+ip^{s}t \ ({\rm mod} \ p^{s+1}) .$$ So $$v_{p}(2^{i\cdot ord_{p}2}-1)=s.$$ Since $$p^s<2^{ord_p2}\leq 2^{p-1}<p^{p-1},$$ we conclude
$s<p-1$.

In particular, by taking $i=(p-1)/ord_p 2$, we have $v_{p}(2^{ord_{p}2}-1)=v_{p}(2^{p-1}-1)$.
\end{proof}

\begin{proposition}\label{v2l}
Let $p$ be an odd prime, and $\hat{\sigma}=(\hat{x}_{1},\hat{x}_{2},\cdots,\hat{x}_{\ell})\subset \mathbb{Z}_{p}^{\times} $ be a periodic orbit of $f: x\mapsto x^2$. If $x_{1}=\hat{x}_{1}+p^{n}\alpha$ for some $\alpha \in \mathbb{Z}_{p}\setminus p\mathbb{Z}_{p}$ and $n\geq 1$, then
$$v_{p}(x_{1}^{2^{\ell r}-1}-1)=n+v_{p}(2^{p-1}-1),$$
where $r=\frac{ord_{p}2}{(\ell, ord_{p}2)}$.
\end{proposition}
\begin{proof} By Theorem \ref{graph} and Proposition \ref{cycleZp},  we have $\ell< p-1$.
Then $$\ell r=\frac{\ell \cdot ord_{p}2}{(\ell, ord_{p}2)} <p \cdot ord_{p}2.$$ Observe that $ord_{p}2\mid\ell r$.
By Lemma \ref{vap}, we have  $$v_{p}(2^{\ell r}-1)=v_{p}(2^{p-1}-1).$$

Let $s=v_{p}(2^{p-1}-1)$, then $2^{\ell r}=1+p^{s}h$ for some $h\in \Z\setminus p\Z$.
Observe that $\hat{x}_{1}\in\Zp^\times$ is a periodic point of $f$ of  period $\ell$. Thus  $\hat{x}_{1}^{2^{\ell}}=\hat{x}_{1}$ and then $\hat{x}_{1}^{2^{\ell r}}=\hat{x}_{1}$. Multiplying $\hat{x}_{1}^{-1}$, we obtain  $$\hat{x}_{1}^{p^sh}=\hat{x}_{1}^{2^{\ell r}-1}=1.$$
So,
\begin{align*}
  x_{1}^{2^{\ell r}-1}-1 &= (\hat{x}_{1}+p^{n}\alpha)^{p^{s}h}-1 \\
   &=  \sum_{i=1}^{p^{s}h}{ p^sh \choose i} \hat{x}_{1}^{p^{s}h-i}\alpha^{i}p^{ni}.\\
\end{align*}
Let $$C_{i}= { p^sh \choose i} \hat{x}_{1}^{p^{s}h-i}\alpha^{i}p^{ni}, \ 1\leq i\leq p^{s}h. $$
Then $v_{p}(C_{1})=n+s$. Moreover,
if $i>s+1$, then $$v_{p}(C_{i})\geq ni>n(s+1)= ns+n\geq s+n.$$
If $1<i\leq s+1$, then by Lemma \ref{vap}, we know that $i<p$. Thus $v_{p}({ p^sh \choose i})=s$. Therefore, $$v_{p}(C_{i})=ni+s>n+s.$$
So,
$$v_{p}(x_{1}^{2^{\ell r}-1}-1)=v_{p}(C_{1})=n+s.$$
\end{proof}

\bigskip
\section{Minimal Decomposition of the square mapping on $\Zp$}\label{decomp}
In this section, we focus on the minimal decomposition of the square mapping on $\mathbb{Z}_{p}$.

By the proof of Lemma \ref{grows-minimal}, if a cycle at a certain level always grows (grows forever) then it will produce a minimal component of $f$. The following proposition shows  when a cycle always grows (grows forever) for the square mapping. A cycle $\sigma$ at level $n$ is said to {\it split $\ell$ times} if $\sigma$ splits, and the lifts of $\sigma$ at level $n+1$ split and inductively all lifts at level $n + j (2\leq j < \ell)$ split.

\begin{proposition}\label{cyclelift}
Let $p$ be an odd prime, and $f: x\mapsto x^2$ be the square mapping. Suppose that $\hat{\sigma}=(\hat{x}_{1},\hat{x}_{2},\cdots,\hat{x}_{\ell})\subset \mathbb{Z}_{p}^{\times}$ is an $\ell$-periodic orbit of $f$. For each $n\geq 1$, let $\sigma_{n}=(x_{1},\cdots,x_{\ell})\subset \mathbb{Z}/p^{n}\mathbb{Z}$ be the $\ell$-cycle of the induced map $f_{n}$ such that
$$\sigma_{n} \equiv  \hat{\sigma} \ \ ({\rm mod} \ p^n), \quad i.e., \text{ for all } 1\leq i \leq \ell,  \ x_i \equiv  \hat{x}_i \ \ ({\rm mod} \ p^n).$$

\noindent {\rm 1)}
If $2^{\ell}\equiv 1 \ ({\rm mod} \ p)$, then $\sigma_{n}$ splits. There is one lift $\sigma_{n+1}$ such that    $\sigma_{n+1}\equiv \hat{\sigma} \ ({\rm mod} \ p^{n+1})$  and all other lifts split $v_{p}(2^{p-1}-1)-1$ times then all descendants at level $n+v_{p}(2^{p-1}-1)$ grow forever.

\noindent {\rm 2)}
If $2^{\ell}\not\equiv 1 \ ({\rm mod} \ p)$, then $\sigma_{n}$ partially splits. Let $\sigma_{n+1}$ be a lift of $\sigma_{n}$.
\begin{itemize}
\item [\rm(a)]If $\sigma_{n+1}$ is the lift of length $\ell$, then $\sigma_{n+1}\equiv \hat{\sigma} \ ({\rm mod} \ p^n),$  and $\sigma_{n+1}$ partially splits.
\item [\rm(b)]If $\sigma_{n+1}$ is a lift of length $\ell r$ for some integer $r>1$, then $r=\frac{ord_{p}2}{(ord_{p}2,\ell)}$  and $\sigma_{n+1}$ split $v_{p}(2^{p-1}-1)-1$ times then all descendants of $\sigma_{n+1}$ at level $n+v_{p}(2^{p-1}-1)$ grow forever.
\end{itemize}
 \end{proposition}

\begin{proof}Let $g=f^{\ell}: x\mapsto x^{2^{\ell}}$ be the $\ell$-th iterate of $f$. Then,
\begin{align*}
 a_{n}(x_1) &= g^{\prime}(x_{1})=2^{\ell}x_{1}^{2^{\ell}-1},\\
 b_{n}(x_{1})&= \frac{x_{1}^{2^{\ell}}-x_{1}}{p^{n}}.
\end{align*}
Since $\hat{x}_{1}^{2^{\ell}}=\hat{x}_{1}$, then $\hat{x}_{1}^{2^{\ell}-1}=1$ and hence $x_{1}^{2^{\ell}-1}\equiv 1\ ({\rm mod} \ p^n)$. Thus, $$a_{n}(x_1)\equiv 2^{\ell}\ ({\rm mod} \ p).$$

\noindent {\rm 1)} Assume  $2^{\ell}\equiv 1 \ ({\rm mod} \ p)$. Then $a_{n}(x_1)\equiv 1\ ({\rm mod} \ p)$.
Let $s=v_{p}(2^{p-1}-1)$. Observe that $ord_{p}2\leq\ell\leq {p-1}$ and $ord_{p}2 \mid \ell$. By Lemma \ref{vap}, we have $v_{p}(2^{\ell}-1)=s\geq1$.
Write $$2^{\ell}=1+p^{s}h$$ for some integer $h$ with $(h,p)=1$.  Since $x_{1}\equiv \hat{x}_{1}\ ({\rm mod} \ p^{n})$, we have $x_{1}=\hat{x}_{1}+p^{n}t$ for some $t\in \mathbb{Z}_{p}$. Thus, by Proposition \ref{v2l} and
  $$g(x_{1})-x_{1} = x_{1}(x_{1}^{2^{\ell}-1}-1),$$
   we deduce that $b_{n}(x_{1})\equiv 0 \ ({\rm mod} \ p)$. Thus $\sigma_{n}$ splits.

Let $\sigma_{n+1}=(y_{1},\cdots,y_{\ell})\subset \mathbb{Z}/p^{n+1}\mathbb{Z}$ be a lift of $\sigma_{n}$. We distinguish the following two cases. \\
  \indent {\rm i)} Assume  $y_{1}\equiv \hat{x}_{1}  \ ({\rm mod} \ p^{n+1})$. Then $\sigma_{n+1} \equiv \hat{\sigma}  \ ({\rm mod} \ p^{n+1})$ and $\sigma_{n+1}$ behaves the same as $\sigma_n$.

  \indent {\rm ii)} Assume  $y_{1}\not\equiv \hat{x}_{1}  \ ({\rm mod} \ p^{n+1})$. Then $y_{1}=\hat{x}_{1}+p^{n}\alpha$ for some $\alpha\in \mathbb{Z}_{p}\setminus p\mathbb{Z}_{p}$. Since $2^{\ell}\equiv 1 \ ({\rm mod} \ p)$, we have $ord_{p}2 \mid \ell$. By
   Proposition  \ref{v2l} and
  $$g(y_{1})-y_{1} = y_{1}(y_{1}^{2^{\ell}-1}-1),$$
   we get
  $$v_{p}(b_{n+1}(y_{1}))=s-1.$$

  If $s=1$, then $\sigma_{n+1}$ grows. By Lemma \ref{cycle-p>3}, the lift of $\sigma_{n+1}$ grows forever.

  If $s>1$, then $\sigma_{n+1}$ splits. By induction, let $\sigma_{n+1+i}$ be a lift of $\sigma_{n+1}$ at level $n+1+i$ for $0\leq i<s-1$, then $\sigma_{n+1+i}$ splits. Let $\sigma_{n+s}=(z_{1},\cdots,z_{\ell})$ be a lift of $\sigma_{n+1}$ at level $n+s$. By Proposition  \ref{v2l},
  $v_{p}(b_{n+s}(z_{1}))=0$, so $\sigma_{n+s}$ grows. By Lemma \ref{cycle-p>3}, the lift of $\sigma_{n+s}$ grows forever.\\

\noindent {\rm 2)} Assume $2^{\ell}\not\equiv 1 \ ({\rm mod} \ p)$. Then  $a_{n}(x_{1})\not\equiv 0,1\ ({\rm mod} \ p)$. Thus $\sigma_{n}$ partially splits. Let $\sigma_{n+1}=(y_{1},\cdots,y_{\ell r})\subset \Z/p^{n+1}\Z$ be a lift of $\sigma_{n}$ of length $\ell r$.

  If $r=1$, then by Proposition \ref{cycleFp}, we get that $\sigma_{n+1}\equiv \hat{\sigma} \mp{n+1},$ and $\sigma_{n+1}$ behaves the same as $\sigma_n$.

  If $r>1$, then $r$ is the order of $a_{n}$ in  $(\Z/p\Z)^*$. By Lemma \ref{ord}, $r=\frac{ord_{p}2}{(ord_{p}2,\ell)}$.
  By Lemma \ref{vap}, we have $$v_{p}(2^{\ell r-1})=v_{p}(2^{p-1})=s.$$ Notice that $y_{1}\equiv \hat{x}_{1} \mp{n}$ and $y_{1}\not\equiv \hat{x}_{1} \mp{n+1}$.  By Proposition  \ref{v2l},
  $$v_{p}(g^{r}(y_{1})-y_{1})=v_{p}(y_{1}(y_{1}^{2^{\ell r-1}}-1))=v_{p}(y_{1}^{2^{\ell r-1}}-1)=n+s.$$
  The same argument as the case ii) of 1) implies that $\sigma_{n+1}$ splits $s-1$ times then all the descendants of $\sigma_{n+1}$ at level $n+s$ grow forever.
\end{proof}

Now we are ready to prove our main result.
\begin{proof}[Proof of Theorem \ref{squringdecompsition}]
 By Theorem \ref{graph}, we know the dynamical structure of $f$ at the first level. Let $\sigma=(x_1,x_2,\cdots, x_{\ell}) \subset  (\Z/p\Z)^*$ be a cycle of length $\ell$ at the first level. By Proposition \ref{cycleZp}, there exists a unique periodic orbit $\hat{\sigma}=(\hat{x}_1,\hat{x}_2,\cdots,\hat{x}_\ell) \subset\X_{\sigma}$ of $f$ with the same length of $\sigma$.

 By Proposition \ref{cyclelift} and Lemmas \ref{minimal-part-to-whole} and \ref{cycle-p>3}, we get the minimal decomposition of system $(\X_{\sigma},f)$:
  $$\X_{\sigma}=\{\hat{x}_{1},\cdots,\hat{x}_{\ell}\}\sqcup \left(\bigsqcup_{n\geq 1}\bigsqcup_{1\leq i\leq \ell}S_{p^{-n}}(\hat{x}_{i})\right),$$
where for each $n\geq 1$, the set $\bigsqcup_{1\leq i\leq \ell}S_{p^{-n}}(\hat{x}_{i})$ consists of   $$\frac{(p-1)\cdot (ord_{p}2,\ell)}{ord_{p}2}\cdot p^{v_{p}(2^{p-1}-1)-1}$$ minimal components and each minimal component consists of $j:=\frac{\ell\cdot ord_{p}2}{(ord_{p}2,\ell)}$ closed disks of radius $p^{-n-v_{p}(2^{p-1}-1)}$.

By Theorem \ref{structure-minimal} and Proposition \ref{cyclelift},  each nontrivial  minimal subsystem (which is not a periodic orbit) of $(\X_{\sigma},f)$ is conjugate to the adding machine on the odometer $\Z_{(p_s)}$, where
  $$(p_s)=(\ell,\ell j,\ell j p,\ell j p^2,\cdots).$$
\end{proof}

%
%
%
%
%

\bigskip
\section{Examples}\label{examples}

Recall that $S_{1}(0)$ is the unit sphere and $f$ is the square mapping.
For different primes, the dynamical behaviors of $(S_{1}(0),f)$ are quite different.

A Fermat prime is a prime number $p$ of the form $p=2^{2^n}+1$ where $n$ is a nonnegative integer.
It is known that the iteration graph of square mapping on $\mathbb{F}^*_{p}$ of the nonzero elements in the finite field $\mathbb{F}_{p}$ is a tree attached to the unique loop of $1$ when $p$ is a Fermat prime, and conversely, if there is only one loop then $p$ must be a Fermat prime (\cite{Rog96}). In this case, $0$ and $1$ are the only fixed points of $f$, the disk $D_{1}(0)$ is the attracting basin of the fixed point $0$. The disk $D_{1}(1)$ is the unique  Siegel disk. Furthermore, we have a minimal decomposition of $D_{1}(1)$  by Theorem \ref{squringdecompsition}. The other open disks with radius $1$ are attracted by the Siegel disk $D_{1}(1)$.
Decompose $D_{1}(1)$ as
$$D_{1}(1)=\{1\}\sqcup \left(\bigsqcup_{i\geq 1}S_{p^{-i}}(1)\right).$$
Then each sphere $S_{p^{-i}}(1)$ consists of  $p^{v_{p}(2^{p-1}-1)-1}$ minimal components, and each minimal component is a union of $p$ closed disks of radius $p^{-i-v_{2}(2^{p-1}-1)}$.

\medskip
An odd prime $p$ is called a Wieferich prime if
 $$2^{p-1}\equiv 1 \mp2.$$
If an odd prime $p$ is not a Wieferich prime, we know that $v_{p}(2^{p-1}-1)=1$.
For a cycle $\sigma=(x_1,x_2,\cdots, x_{\ell}) \subset  \Z/p\Z$ of length $\ell$ at the first level,  Proposition \ref{cycleZp} implies that there exists a unique periodic orbit $\hat{\sigma}=(\hat{x}_1,\hat{x}_2,\cdots,\hat{x}_\ell) \subset\X_{\sigma}$ of $f$ with the same length of $\sigma$. By Proposition \ref{cyclelift}, the lifts which do not correspond to the periodic orbit grow forever and the  lift corresponding to the periodic orbit behaves the same as $\sigma$. Thus for each integer $n\geq 1$,  the union $\bigsqcup_{1\leq i\leq \ell}S_{p^{-n}}(\hat{x}_{i})$ of the spheres consists of   $\frac{(p-1)\cdot (ord_{p}2,\ell)}{ord_{p}2}$ minimal components and each minimal component consists of $\frac{\ell\cdot ord_{p}2}{(ord_{p}2,\ell)}$ closed disks of radius $p^{-n-1}$.

The only known Wieferich primes $1093$ and $3511$ were found by Meissner in 1913 and Beeger in 1922, respectively. It has been conjectured that only finitely many Wieferich primes exist. Silverman \cite{Silverman88} showed in 1988 that if the abc conjecture holds, then there exist infinitely many non-Wieferich primes. Numerical evidence suggests that very few of the prime numbers in a given interval are Wieferich primes. A proof of the abc conjecture would not automatically prove that there are only finitely many Wieferich primes, since the set of Wieferich primes and the set of non-Wieferich primes could possibly both be infinite and the finiteness or infiniteness of the set of Wieferich primes would have to be proven separately.

For the known Wieferich primes $p=1093$ or $3511$, we have $v_{p}(2^{p-1}-1)=2$. For a cycle $\sigma=(x_1,x_2,\cdots, x_{\ell}) \subset  \Z/p\Z$ of length $\ell$ at the first level.  Similar to the general case, there exists a unique periodic orbit $\hat{\sigma}=(\hat{x}_1,\hat{x}_2,\cdots,\hat{x}_\ell) \subset\X_{\sigma}$ of $f$ with the same length of $\sigma$. Different to the non-Wieferich primes, the lifts which do not correspond to the periodic orbit split one time at first and then  all the descendants grow forever. For each integer $n\geq 1$,  the union $\bigsqcup_{1\leq i\leq \ell}S_{p^{-n}}(\hat{x}_{i})$ of the spheres consists of   $\frac{p(p-1)\cdot (ord_{p}2,\ell)}{ord_{p}2}$ minimal components and each minimal component consists of $\frac{\ell\cdot ord_{p}2}{(ord_{p}2,\ell)}$ closed disks of radius $p^{-n-2}$.

However, the existence of prime number $p$ such that  $v_{p}(2^{p-1}-1)> 2$ is still unknown.

%
%
%
%

%
%
%
%
%
%
%
%
%
\bigskip
\section*{Acknowledgement}
Shilei Fan was partially supported by self-determined research funds of CCNU from the colleges¡¯ basic research and operation of MOE (Grant No. CCNU14Z01002) and NSF of China (Grant No. 11231009). Lingmin Liao was partially supported by 12R03191A - MUTADIS (France) and the project PHC Orchid of MAE and MESR of France.

\bigskip
\bibliographystyle{plain}


\end{document}